\documentclass[11pt,reqno]{amsart}
\usepackage{amsmath}
\usepackage{amsfonts}
\usepackage{amssymb}
\usepackage{mathrsfs}
\usepackage{cite}

\newcommand{\wh}{\widehat}

\newcommand{\Thmname}{Theorem}
\newcommand{\Propname}{Proposition}
\newcommand{\Lemmaname}{Lemma}
\newcommand{\Definitionname}{Definition}
\newtheorem{Thm}{\Thmname}

\newtheorem{Lemma}[Thm]{\Lemmaname}
{\theoremstyle{definition}
}
{\theoremstyle{remark}
}

{\theoremstyle{remark}
}

\newtheorem*{Lemma*}{Lemma}

\title{A simple proof of Brown's diagonalizability theorem}

\author{Benjamin Steinberg}
\address{School of Mathematics and Statistics\\
Carleton University \\
1125 Colonel By Drive\\
Ottawa, Ontario  K1S 5B6 \\
Canada}
\thanks{The author was supported in part by NSERC}
\email{bsteinbg@math.carleton.ca}
\date{October 10, 2009}

\begin{document}
\maketitle

We present here a simple proof of Brown's diagonalizability theorem for certain elements of the algebra of a left regular band~\cite{Brown1,Brown2}, including probability measures.  Brown's theorem also provides a uniform explanation for the diagonalizability of certain elements of Solomon's descent algebra, since the descent algebra embeds in a left regular band algebra~\cite{Brown1,Brown2}.  Recall that a left regular band is a semigroup satisfying the identities $x^2=x$ and $xyx=xy$.  In this paper all semigroups are assumed finite.

Let $S$ be a left regular band with identity (there is no loss of generality in assuming this) and let $L$ be the lattice of principal left ideals of $S$ ordered by inclusion\footnote{Brown calls the dual of this lattice the support lattice.}.  We view $L$ as a monoid via its meet, which is just intersection. There is a natural surjective homomorphism $\sigma\colon S\to L$, called the \emph{support map}, given by $\sigma(s) = Ss$.  A key fact that we shall exploit is that $\sigma(s)\leq \sigma(t)$ if and only if $st=s$, that is, $s\in St$ if and only if $st=s$. Indeed, let $S$ act on the right of itself. Because $t$ is an idempotent, it acts as the identity on its image; but this is just $St$.

Let $k$ be a field and let
\begin{equation}\label{w}
w=\sum_{t\in S}w_tt\in kS.
\end{equation}
For $X\in L$, define
\begin{equation}\label{lambdas}
\lambda_X=\sum_{\sigma(t)\geq X}w_t.
\end{equation}
Brown~\cite{Brown1,Brown2} showed that $k[w]$ is split semisimple provided that $X>Y$ implies  $\lambda_X\neq \lambda_Y$.  We give a new proof of this by showing that if $\lambda_1,\ldots,\lambda_k$ are the distinct elements of $\{\lambda_X\mid X\in L\}$, then
\begin{equation}\label{minpol}
0=\prod_{i=1}^k(w-\lambda_i).
\end{equation}
This immediately implies that the minimal polynomial of $w$ has distinct roots and hence $k[w]$ is split semisimple.

Everything is based on the following formula for $sw$.

\begin{Lemma}\label{theformula}
Let $s\in S$.  Then
\[sw = \lambda_{\sigma(s)}s + \sum_{\sigma(t)\ngeq \sigma(s)} w_tst\]
and moreover, $\sigma(s)>\sigma(st)$ for all $t$ with $\sigma(t)\ngeq \sigma(s)$.
\end{Lemma}
\begin{proof}
Using that $\sigma(t)\geq \sigma(s)$ implies $st=s$, we compute
\begin{align*}
sw &= \sum_{\sigma(t)\geq \sigma(s)}w_tst+\sum_{\sigma(t)\ngeq \sigma(s)}w_tst\\
   &= \sum_{\sigma(t)\geq \sigma(s)}w_ts+\sum_{\sigma(t)\ngeq \sigma(s)}w_tst\\
   &= \lambda_{\sigma(s)}s+\sum_{\sigma(t)\ngeq \sigma(s)}w_tst.
\end{align*}
It remains to observe that $\sigma(t)\ngeq \sigma(s)$ implies  $\sigma(st) = \sigma(s)\sigma(t)< \sigma(s)$.
\end{proof}

The proof of \eqref{minpol} proceeds via an induction on the support. Let us write $\wh 0$ for the bottom of $L$ and $\wh 1$ for the top.  If $X\in L$, put
\[\Lambda_X = \{\lambda_Y\mid Y\leq X\}\ \text{and}\ \Lambda'_X=\{\lambda_Y\mid Y<X\}.\]
Our hypothesis says exactly that $\Lambda_X = \{\lambda_X\}\mathrel{\dot{\cup}} \Lambda'_X$ (disjoint union).  Define polynomials $p_X(z)$ and $q_X(z)$, for $X\in L$, by
\begin{align*}
p_X(z) &= \prod_{\lambda_i\in \Lambda_X}(z-\lambda_i)\\
q_X(z) &= \prod_{\lambda_i\in \Lambda_X'}(z-\lambda_i)=\frac{p_X(z)}{z-\lambda_X}.
\end{align*}
Notice that, for $X>Y$, we have $\Lambda_Y\subseteq \Lambda'_X$, and hence $p_Y(z)$ divides $q_X(z)$, because $\lambda_X\notin \Lambda_Y$ by assumption.  Also observe that  \[p_{\wh 1}(z) = \prod_{i=1}^k(z-\lambda_i)\] and hence establishing \eqref{minpol} is equivalent to proving $p_{\wh 1}(w)=0$.

\begin{Lemma}\label{kill}
If $s\in S$, then $s\cdot p_{\sigma(s)}(w)=0$.
\end{Lemma}
\begin{proof}
The proof is by induction on $\sigma(s)$ in the lattice $L$.  Suppose first $\sigma(s)=\wh 0$; note that $p_{\wh 0}(z)=z-\lambda_{\wh 0}$. Then since $\sigma(t)\geq \sigma(s)$ for all $t\in S$, Lemma~\ref{theformula} immediately yields $s(w-\lambda_{\sigma(s)}) = 0$.   In general, assume the lemma holds for all $s'\in S$ with $\sigma(s')<\sigma(s)$.  Then by Lemma~\ref{theformula}
\begin{align*}
s\cdot p_{\sigma(s)}(w)= s\cdot (w-\lambda_{\sigma(s)})\cdot q_{\sigma(s)}(w)
= \sum_{\sigma(t)\ngeq \sigma(s)}w_tst\cdot q_{\sigma(s)}(w)
=0.
\end{align*}
Here the last equality follows because $\sigma(t)\ngeq \sigma(s)$ implies $\sigma(s)>\sigma(st)$ and so $p_{\sigma(st)}(z)$ divides $q_{\sigma(s)}(z)$, whence induction yields $st\cdot q_{\sigma(s)}(w)=0$.
\end{proof}

Applying the lemma to the identity element of $S$ yields $p_{\wh 1}(w)=0$ and hence we have proved:

\begin{Thm}
Let $w$ be as in \eqref{w} and let $\lambda_X$ be as in \eqref{lambdas} for $X\in L$.  If $X>Y$ implies $\lambda_X\neq \lambda_Y$, then $k[w]$ is split semisimple.
\end{Thm}

If $k=\mathbb R$, and $w$ is a probability measure, then $X>Y$ implies $\lambda_X>\lambda_Y$ provided the support of $w$ generates $S$ as a monoid.  If this is not the case, then semisimplicity of $\mathbb R[w]$ follows by considering $\mathbb R[w]\subseteq \mathbb RT\subseteq \mathbb RS$ where $T$ is the submonoid generated by the support of $w$.

\bibliographystyle{abbrv}
\bibliography{standard2}
\end{document}